\documentclass[11pt]{article}
\usepackage{rotating,multirow,amsmath,amssymb,fullpage}

\usepackage{natbib}
 \bibpunct[, ]{(}{)}{,}{a}{}{,}%

\newtheorem{definition}{Definition}
\newtheorem{theorem}{Theorem}
\newtheorem{lemma}{Lemma}
\newtheorem{proposition}{Proposition}
\newenvironment{proof}{\paragraph{Proof:}}{\hfill$\Box$}

\begin{document}
\title{The expanding search ratio of a graph}
% \thanks{A preliminary version of this paper appeared in the Proceedings of the
% 		33rd International Symposium on Theoretical Aspects of Computer Science (STACS), 2016.}}

% Block of authors and their affiliations starts here:
% NOTE: Authors with same affiliation, if the order of authors allows,
%   should be entered in ONE field, separated by a comma.
%   \EMAIL field can be repeated if more than one author
\author{%
Spyros Angelopoulos%
\thanks{Sorbonne Universit\'es, UPMC Univ Paris 06, CNRS, LIP6, Paris, France}
\and
Christoph D\"urr%
\footnotemark[1]
\and
Thomas Lidbetter%
\thanks{Department of Management Science and Information Systems, Rutgers Business School, Newark, NJ, USA}
}

\maketitle

\begin{abstract}
    We study the problem of searching for a hidden target in an environment that is modeled by an edge-weighted graph. A sequence of edges is chosen starting from a given {\em root} vertex such that each edge is adjacent to a previously chosen edge. This search paradigm, known as {\em expanding search} was recently introduced by \cite{AL:expanding} for modeling problems such as searching for coal or minesweeping in which the cost of re-exploration is negligible. It can also be used to model a team of searchers successively splitting up in the search for a hidden adversary or explosive device, for example. We define the {\em search ratio} of an expanding search as the maximum over all vertices of the ratio of the time taken to reach the vertex and the shortest-path cost to it from the root.
    This can be interpreted as a measure of the multiplicative regret incurred in searching, and similar objectives have previously been studied in the context of conventional (pathwise) search.
    In this paper we address algorithmic and computational issues of minimizing the search ratio over all expanding searches, for a variety of search environments,
    including general graphs, trees and star-like graphs.  Our main results focus on the problem of finding the {\em randomized expanding search} with minimum {\em expected search ratio}, which is equivalent to solving a zero-sum game between a {\em Searcher} and a {\em Hider}. We solve these problems for certain classes of graphs, and obtain constant-factor approximations for others.
\end{abstract}

%!TEX root = exp search ratio - revision v2.tex

%%%%%%%%%%%%%%%%%%%%%%%%%%%%%%%%%%%%%%%%%%%%%%%%%%%%%%%%%%%%%%%%%%%%%%

%Samples of sectioning (and labeling) in OPRE
%NOTE: (1) \section and \subsection do NOT end with a period
%(2) \subsubsection and lower need end punctuation
%(3) capitalization is as shown (title style).
%\section{Introduction.}\label{intro} %%1.
%\subsection{Duality and the Classical EOQ Problem.}\label{class-EOQ} %% 1.1.
%\subsection{Outline.}\label{outline1} %% 1.2.
%\subsubsection{Cyclic Schedules for the General Deterministic SMDP.}
%\label{cyclic-schedules} %% 1.2.1
%\section{Problem Description.}\label{problemdescription} %% 2.

\section{Introduction}

We consider the problem faced by a {\em Searcher} of locating a stationary target or {\em Hider} located at a vertex of a connected edge-weighted graph $G$. We interpret the weight of an edge as the time taken to search that edge. The search must start at a given vertex $O$ called the {\em root} and consists of a sequence of edges chosen in such a way that every edge must be adjacent to some previous edge, so that the set of edges that have been searched at any point forms a connected subgraph of $G$. For a given search and a given vertex $v$ at which the Hider is located, the {\em search time} of $v$ is the time taken to search all the edges up to and including the first edge that is incident to $v$.

This paradigm of search, recently introduced by \cite{AL:expanding} is known as {\em expanding search}, in contrast to the more usual search paradigm, referred to here as {\em pathwise search} in which a search corresponds to a walk in a graph. The expanding search paradigm is an appropriate model for situations in which the cost of ``re-exploration'' is negligible compared to the cost of searching, for example when mining coal: here digging into a new site is far more costly than moving the drill through an area that has already been dug. Another situation to which this principle applies is securing a dangerous area from hidden explosives; once the area is deemed clear, the searchers can navigate through it at a much lower cost. An alternative interpretation, described in detail in \cite{AL:expanding}, is that of a team of searchers splitting up in the search for a target.

We illustrate the concept of expanding search on a graph with an example. Consider the graph $G$ depicted in Figure \ref{fig:example} with root $O$; vertices $A$, $B$, $C$ and $D$; and edges $OA$, $OB$, $BC$ and $BD$ of lengths $3$, $2$, $2$ and $1$, respectively. An example of an expanding search on $G$, which we will denote by $S$, is the sequence of edges, $OB$, $OA$, $BD$, $BC$. Under $S$, the search time of vertex $D$ is $2+3+1=6$.
\begin{figure}
    \centerline{\includegraphics[width=0.4\textwidth]{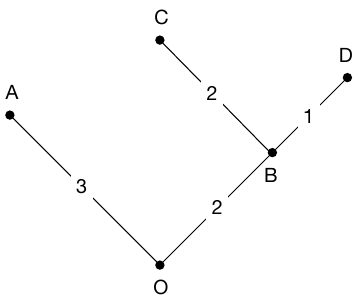}}
    \caption{The graph $G$.}
    \label{fig:example}
\end{figure}

\cite{AL:expanding} take the approach of seeking randomized search strategies that minimize the
expected search time in the worst case: that is, the maximum expected search time over all vertices. They also consider the problem of determining the search that minimizes some weighted average of the search times of the vertices. In this paper, we take an alternative approach by considering a normalized version of the search time obtained by dividing the search time of a vertex $v$ by the length of the shortest path from $O$ to $v$. For example, in the graph $G$ depicted in Figure \ref{fig:example}, the normalized search time under $S$ of vertex $D$ is $6/3 = 2$ since $D$
is at a distance of $2+1=3$ from the root. The maximum the normalized search time takes over all vertices of the graph is called the {\em search ratio}. In $G$, the normalized search time of $S$ is maximized at $B$, where it is equal to $(3+2)/2=2.5$, so this is the search ratio of $S$. This paper studies the problem of finding an expanding search with minimum search ratio.

Our choice of the search ratio as the objective for expanding search is motivated by earlier work by \cite{koutsoupias:fixed},
who introduced this objective in the context of pathwise search. Their approach is analogous to the competitive analysis of online algorithms, in which the performance of an online algorithm is measured against the performance of an optimal offline algorithm; more precisely, the optimal offline algorithm corresponds to simply taking the shortest path to the target.  As in \cite{koutsoupias:fixed}, we consider not only deterministic searches, but also randomized searches, with the aim of finding the randomized expanding search that minimizes the expected value of the normalized search time in the worst case. Equivalently, we view this problem as a zero-sum game between a Hider who chooses a vertex of the graph and a Searcher who chooses an expanding search. The payoff, which the Hider seeks to maximize and the Searcher to minimize, is the normalized search time. This puts our work in the broader category of {\em search games}, a more general framework for games played between a Hider who chooses a point in some search space and a Searcher who makes some choice of how to navigate through the space with the aim of minimizing a given cost function.

It is worth mentioning that normalized cost formulations very similar to the search ratio
have also been previously studied in the context of searching in unbounded domains
(see, e.g., the early work of \cite{beck:yet.more} on the linear search problem as well as the work of \cite{gal:general}
in the context of star search). In such domains, the Hider can ensure the search time is arbitrarily large by choosing positions
arbitrarily far from the root. This observation motivates the need for normalizing the search cost, which is accomplished by dividing this cost by the
shortest-path cost from the root to the Hider.

In the spirit of the work of \cite{koutsoupias:fixed}, in this paper we focus on computational and algorithmic issues of expanding search. We note that \cite{AL:expanding} follow a purely mathematical approach to analyzing expanding search, with an emphasis on evaluating the value of the corresponding zero-sum games; computational and algorithmic issues are not considered. Table~\ref{table:summary} illustrates
the context of our work with respect to previous work.
We note that the problem of minimizing the average search time of the vertices of a graph assuming the pathwise search formulation is precisely the
well-known problem of minimizing the {\em latency} of a graph, also known as the {\em Traveling Repairman} problem
(see~\cite{DBLP:conf/stoc/BlumCCPRS94,DBLP:journals/mp/GoemansK98,DBLP:journals/siamcomp/AroraK03,DBLP:conf/ipco/Sitters02}
for some representative results on this problem). The problem of choosing the randomized (pathwise) search that minimizes the {\em maximum expected search time} of points of a network was formalized by \cite{Gal:79}, and has been extensively studied, as discussed in Subsection~\ref{sec:related}.

\begin{table}[htb!]
    \centering
    \caption{Previous work and relations between search paradigms and objectives.}
    \label{table:summary}
    \begin{tabular}{l|l|l|l|l|}
    \multicolumn{2}{c}{} &
         \multicolumn{3}{c}{\textbf{Objective}} \\ \cline{2-5}
         && \em{Average search time} & \em{Maximum expected search time} & \em{Search ratio} \\
     \cline{2-5}
     \multirow{2}{1em}{\begin{sideways}\textbf{Paradigm}\end{sideways}} &
          \begin{sideways}\emph{Pathwise}\hspace*{.8em}\end{sideways}
                            & \begin{minipage}[b]{45mm}Min. Latency problem\\ \footnotesize((\cite{DBLP:conf/stoc/BlumCCPRS94})\end{minipage}
                            & \begin{minipage}[b]{47mm}Gal's search game\\  \footnotesize((\cite{Gal:79})\end{minipage}
                            & \begin{minipage}[b]{45mm}Searching a fixed graph\\ \footnotesize((\cite{koutsoupias:fixed})\end{minipage} \\
           \cline{2-5}
           &
            \begin{sideways}\emph{Expanding}\hspace*{.8em}\end{sideways}
                           & \begin{minipage}[b]{45mm}Expanding search\\ \footnotesize(\cite{AL:expanding})\end{minipage}
                           & \begin{minipage}[b]{47mm}Expanding search \\\footnotesize((\cite{AL:expanding})\end{minipage}
                           & \begin{minipage}[b]{45mm}This work\\\footnotesize~\end{minipage}\\
          \cline{2-5}
    \end{tabular}
\end{table}

% \begin{table}[htb!]
%   \centering
%   \caption{Previous work and relations between search paradigms and objectives.}
%   \label{table:summary}
%   \begin{tabular}{|l|l|l|l|}
%       \hline
%        & \multicolumn{3}{c|}{\textbf{Objective}} \\ \hline
%        & \em{Average} & \em{Maximum expected} & \em{Search ratio} \\
%        & \em{search time} & \em{search time} &  \\ \hline
%         \textbf{Pathwise} & Min. Latency problem & Gal's search game & Searching a fixed graph\\
%         \textbf{search} & (\cite{DBLP:conf/stoc/BlumCCPRS94}) & (\cite{Gal:79}) & (\cite{koutsoupias:fixed})\\ \hline
%       \textbf{Expanding} &  Expanding search& Expanding search & This work\\
%        \textbf{search} &  (\cite{AL:expanding})& (\cite{AL:expanding}) & \\  \hline
%   \end{tabular}
% \end{table}

\subsection{Related Work} \label{sec:related}
Search theory (and its computational counterpart) has a very rich history of
research. We give a summary of some results that are pertinent to this work.

Following the formalization of network search games by \cite{Gal:79} in the framework of pathwise search
with un-normalized search time, the problem has had considerable attention, for example in \cite{Reijnierse-Potters}, \cite{Pavlovic} and \cite{Gal:2001}. In the latter work the solution of the game was found for all {\em weakly Eulerian} networks. Recent variations on Gal's original game include a setting in which the Searcher chooses his own starting point in \cite{DG:arbitrary} and \cite{ABG:arbitrary}, and the setting in which the Hider is restricted to choosing vertices that have search costs in \cite{BK:search-costs13} and \cite{BK:search-costs15}.

Expanding search was introduced by \cite{AL:expanding} in the setting in which the payoff is the total (un-normalized) cost of finding the Hider. Among other results, Alpern and Lidbetter solved the game in the case that the network is either a tree or $2$-edge-connected. This model was extended by \cite{Lidbetter:multiple} to a setting in which the Searcher must locate multiple hidden objects.

Much of the search games literature been purely mathematical, with less emphasis on issues of complexity, a notable exception being the work of \cite{vonStengel}.
The search ratio of pathwise search was studied in \cite{koutsoupias:fixed}, who showed that the problem of computing the optimal search ratio in a given undirected graph is NP-complete (and MAX-SNP hard to approximate). They also gave a search strategy based on repeated executions of depth-first searches with geometrically increasing depths that achieves a constant approximation of the (deterministic) competitive ratio. Similar results can be obtained concerning the {\em randomized competitive ratio} (assuming that the Searcher randomizes over its strategy space).  Connections between graph searching and other
classic optimization problems such as the Traveling Salesman problem and the Minimum Latency problem
were shown by \cite{DBLP:conf/ciac/AusielloLM00}. The setting in which the search graph is
revealed as the search progresses was studied by \cite{fleischer:online}. The latter also addressed connections between searching and exploring an environment, where the latter operation is defined as moving around the environment until all possible hiding positions are ``visible'' (the formal definition of visibility depends on the particular environment).

A specific search environment that has attracted considerable attention in the search literature is
the {\em star-like} environment. More specifically,
in the unbounded variant, the search domain consists of a set of infinite lines which have
a common intersection point (the root of the Searcher); this problem is also known as {\em ray searching}.
%In the bounded variant, the search graph is a finite, edge-weighted star that consists of $m+1$ vertices and $m$ edges.
Ray searching is a natural generalization of the well-known {\em linear search} problem introduced independently by \cite{beck:ls} and \cite{bellman} (informally called the ``cow-path problem''). Optimal strategies for linear search under the (deterministic) competitive
ratio were first given by \cite{beck:yet.more}.
\cite{gal:minimax} gave optimal strategies for the generalized problem of ray searching, a result that was rediscovered later by computer scientists (see \cite{yates:plane}).
Other related work includes the study of
randomization by \cite{schuierer:randomized} and \cite{ray:2randomized}; multi-Searcher strategies by \cite{alex:robots};
searching with turn cost by \cite{demaine:turn};
the variant in which some probabilistic information on target placement is
known by \cite{jaillet:online} and \cite{informed.cows}; and
the related problem of designing {\em hybrid algorithms} by \cite{hybrid}.

Bounded star search, namely the case in which an upper bound is known on the
distance of the target from the root was studied in  \cite{ultimate} and \cite{revisiting:esa}.
New performance measures that are applicable in the context of multi-target searching were introduced
by \cite{hyperbolic} and \cite{oil}
(i.e., the setting in which there are more than one Hider and the Searcher must locate one of them).
The problem of locating a certain number among the many Hiders was studied by \cite{multi-target}.

It must be emphasized that star search has applications that are not necessarily
confined to the concept of locating a target (which explains its significance and popularity).
Indeed star search offers
an abstraction that applies naturally in settings in which we seek an intelligent
allocation of resources to tasks. More precisely,
it captures decision-making aspects when the objective is
to successfully complete at least one task,
without knowing in advance the completion time of each task. Some concrete applications
include: drilling for oil in a number of different locations in \cite{oil}; the
design of efficient {\em interruptible} algorithms, i.e., algorithms that return acceptable solutions even if interrupted during their
execution in \cite{BPZF.2002.scheduling} and \cite{spyros:ijcai15}; and database query optimization
(in particular, pipelined filter ordering in \cite{Condon:2009:ADA:1497290.1497300}). We discuss the latter work in more detail in Section~\ref{sec:star}.

% here the problem can be formulated as a star-search problem, in which each ray of the star
% corresponds to a different location that is a candidate for drilling, and the search cost is the overall time up to the location
% of the oil source. As a second application, star search offers valuable insights into the design of efficient {\em interruptible}
% algorithms; namely, in the design of algorithms that return acceptable solutions even if interrupted during their
% execution. This is an area that is well-studied in the field of Artificial Intelligence (see, e.g.~\cite{BPZF.2002.scheduling}
% and the recent work~\cite{spyros:ijcai15}).

\subsection{Contribution}
In this work we study expanding search under the search-ratio measure,
assuming a variety of search spaces such as stars, trees, and general edge-weighted, undirected graphs.
Our main motivation is to explore how the transition from pathwise to expanding search affects the
deterministic and the randomized search ratios.

We begin in Section~\ref{sec:preliminaries} with the definitions of the
(expanding) search ratio and randomized search ratio. In Section \ref{sec:general.discrete} we show that the problem of finding the optimal (deterministic) search ratio is NP-hard (using a substantially more complicated reduction than for pathwise search in \cite{koutsoupias:fixed}). Applying well-known iterative deepening techniques, we obtain a $4\ln(4) \approx 5.55$ approximation.

Our main technical results, presented in Section \ref{sec:trees}, apply to the setting where the graph is an unweighted graph or a
weighted tree. Here, it is easy to show that an optimal deterministic search strategy searches the vertices in non-decreasing order of
distance from the root (and chooses the corresponding edges accordingly).
This strategy is also a $2$-approximation of the randomized search ratio. To see why the randomized search ratio might as little as half of the deterministic search ratio, suppose two vertices are at approximately the same distance from the root. Then it is possible that by using randomization, the expected search times of the vertices can be ``smoothed out'', which may decrease the randomized search ratio. Therefore, we define a randomized search strategy that approximates the randomized search ratio within a factor of $5/4$, representing a significant improvement over the aforementioned $2$-approximation. The idea of the strategy is to choose a subtree containing nodes within some randomly chosen radius of the root, search it, contract this subtree to the root, and repeat. The method of searching each of these subtrees is by what we call a {\em Random Depth-First Search}, which is an equiprobable choice of a depth-first search $S$ and the depth-first search that arrives at the leaves of the trees in the reverse order to $S$. Thus vertices of the graph at a similar distance to each other are reached at roughly the same time, on average. Improved approximations via randomization are usually not easy to achieve (see, e.g. \cite{koutsoupias:fixed}). Our result confirms the intuitive expectation that randomization has significant benefits.

In Section~\ref{sec:star} we study the problem of bounding the randomized search ratio of a weighted graph, as function of the number of its vertices. 
First, we argue that in the case of a star graph, the setting is equivalent to a problem considered by~\cite{Condon:2009:ADA:1497290.1497300} in the context of pipelined filter ordering in database query optimization (though their problem is not explicitly described as a search game). For this special case, they presented an algorithm for constructing the optimal randomized search strategy and an expression for the randomized search ratio. We show that the results of~\cite{Condon:2009:ADA:1497290.1497300} imply that the randomized search ratio of a star of $n$ vertices plus the root cannot exceed $(n+1)/2$; furthermore, we show that the same result applies to general weighted graphs. 
We note, however, that the result of~\cite{Condon:2009:ADA:1497290.1497300} follows from a complicated flow-based algorithm, which does not readily offer intuition about why the uniform star has the maximum randomized search ratio. For this reason, we provide an alternative proof of the upper bound in star graphs (which we obtained independently in the conference version of the paper~\cite{stacs-expanding}), and which is based on the analysis of a simple, intuitive search strategy using game-theoretic techniques. 

As argued earlier, star-search problems have applications that transcend searching. This is indeed the case in expanding search. Consider the following problem: we are given a collection of $n$ boxes, among which only one contains a prize. We can open a box $i$ at cost $d_i$. We seek a (randomized) strategy for locating the prize, and the randomized search ratio of the strategy is the total expected cost of all opened boxes, divided by the cost of the box that holds the prize. This problem is equivalent to the problem of finding the (randomized) search ratio of a star graph.

We may also interpret expanding search on a tree as the scheduling of jobs with precedence constraints, where each vertex $v$ of the tree corresponds to a job whose processing time is the length of the edge immediately preceding $v$, and all jobs on the path from $v$ to the root must be executed before $v$ can be executed. An expanding search corresponds to a feasible schedule, and the distance of $v$ from the root corresponds to the minimum possible completion time of $v$ over all choices of schedule, which we can interpret as the offline cost of completing $v$. We can then consider the problem of choosing a schedule to minimize the maximum ratio of the completion time of a job to its ``offline cost''. This is exactly our expanding search problem.

Since our main objective is to study the algorithmic and computational impact of {\em re-exploration}
due to the transition from pathwise search to expanding search, it is important to compare our results to the best-known bounds
in the context of pathwise search. More precisely, for unweighted graphs, \cite{koutsoupias:fixed} gives asymptotic approximations of the deterministic
and randomized search ratios equal to 6 and 8.98, respectively, but its techniques appear to be applicable also to general graphs, at the expense of somewhat
larger, but nevertheless constant approximations. Furthermore, \cite{koutsoupias:fixed} note
that the problems of computing the search ratios of trees are  ``surprisingly hard''.
In contrast, for expanding search of unweighted graphs and (weighted) trees we obtain optimal algorithms and a 5/4 (asymptotic) approximation of the
deterministic and randomized search ratios, respectively. We thus demonstrate that the transition from pathwise to expanding search can yield
dramatic improvements in terms of the approximability of the search ratios.
For general graphs, we note that our $5.55$ approximation is strict, and not asymptotic. As a last observation, we note that
the pathwise and expanding search algorithms appear to depend crucially on the approximability of the Traveling Salesman problem and the Steiner Tree problem, respectively.

% We summarize our main results in Table \ref{tab:results}.
% \begin{table}[htb]
% \caption{Main results.}
% \begin{center}
% \begin{tabular}{|l|c|c|}
% \hline
% & {\bf Deterministic search ratio} & {\bf Randomized search ratio}\\ \hline
% {\bf Unweighted graphs} & Optimal strategy & $1.25$-approximation \\
% {\bf and trees}&   & \\ \hline
% {\bf Weighted graphs} & 5.55-approximation & - \\
% \hline
% \end{tabular}
% \label{tab:results}
% \end{center}
% \end{table}

\section{Preliminaries}
\label{sec:preliminaries}

Let $G=(\mathcal{E},\mathcal{V})$ be an undirected, connected, edge-weighted graph with $|\mathcal{V}|=n+1$,
and a distinguished root vertex $O \in \mathcal{V}$. The weight or {\em length} of edge $e \in \mathcal{E}$, denoted by $\lambda(e)$, represents the time required to search that edge (we assume, via normalization, that $\lambda(e) \geq 1$ for all edges $e$). For subgraphs or subsets of edges $X$, we write $\lambda(X)$ for the sum of the lengths of all the edges in $X$. We will call a graph of unit edge weights {\em unweighted}, otherwise it is weighted.

An expanding search, or simply \emph{search strategy} on $G$ is a sequence of edges, starting from the root, chosen so that the set of edges that have been searched at any given point in the sequence is a connected, increasing set. More precisely:
\begin{definition}
    \label{def:discrete-exp}
    An expanding search $S$ on a graph $G$ with root vertex $O$ is a sequence of edges $(e_1,\ldots,e_{n})$ such that every {\em prefix} $\{e_1,\ldots,e_k\}$, $k=1,\ldots,n$ is a subtree of $G$ rooted at $O$. We denote the set of all expanding searches on $(G,O)$ by $\mathcal{S}=\mathcal{S}(G,O)$.
\end{definition}

We note that if we wished,  we could define search strategies less restrictively so that every prefix is simply a connected subgraph rather than a tree,
but it will soon be clear that strategies fulfilling Definition \ref{def:discrete-exp} are dominant.

For a given vertex $v \in \mathcal{V}$ and a given search strategy $S=(e_1,\ldots,e_{n})$, denote by $S_v$ the first prefix $\{e_1,\ldots,e_k\}$ that covers $v$. The \emph{search time}, $T(S,v)$ of $v$ is the total time $\lambda(S_v)$ taken to search all the edges before $v$ is discovered. Let $d(v)$ denote the length of the shortest path from $O$ to $v$, which is the minimum time for the Searcher to discover $v$. For $v \neq O$ the {\em normalized search time} is denoted by $\hat{T}(S,v)=T(S,v)/d(v)$.

\begin{definition} The (deterministic) search ratio $\sigma_S=\sigma_S(G)$ of a search strategy $S$ for the graph $G$ is defined as
    \[
    \sigma_S(G) = \max_{v \in \mathcal{V}-\{O\}} \hat{T}(S,v).
    \]
    The (deterministic) search ratio, $\sigma = \sigma(G)$ of $G$ is defined as
    \[
    \sigma(G) = \min_{S \in \mathcal{S}} \sigma_S(G).
    \]
    If $\sigma_S=\sigma$ we say $S$ is optimal.
\end{definition}
\smallskip

We will also consider \emph{randomized search strategies}, that is some probabilistic choice of search strategies.
Following the standard notation, we denote randomized strategies by lower case letters, and for a randomized search strategy $s$ and a vertex $v$, we extend the notation $T(s,v)$ to denote the expected search time of $v$. Similarly we write $\hat{T}(s,v)$ for the {\em expected normalized search time} $T(s,v)/d(v)$.

\begin{definition} \label{def:rho} The randomized search ratio $\rho_s=\rho_s(G)$ of a randomized search strategy $s$ for the graph $G$ is given by
    \[
    \rho_s(G) = \max_{v \in \mathcal{V}-\{O\}} \hat{T}(s,v).
    \]
    The randomized search ratio, $\rho=\rho(G)$ of $G$ is given by
    \[
    \rho(G) = \inf_s \rho_s(G),
    \]
    where the infinum is taken over all possible randomized search strategies $s$. If $\rho_s=\rho$ we say $s$ is optimal.
\end{definition}
\smallskip

We will view the randomized search ratio $\rho$ through the lens of a finite zero-sum game between a Searcher and a malevolent Hider. The Searcher's pure strategy set is the set $\mathcal{S}$ of expanding searches and the Hider's pure strategy set is the set $\mathcal{V}-\{O\}$ of non-root vertices of $G$. For a Hider strategy $v \in \mathcal{V}-\{O\}$ and a Searcher strategy $S \in \mathcal{S}$, the payoff of the game is $\hat{T}(S,v)$, which the Hider wishes to maximize and the Searcher wishes to minimize. Since the strategy sets are finite, the game has a value and optimal mixed strategies for both players. By the standard minimax theorem for zero-sum games, the value of the game is equal to the randomized search ratio and an optimal randomized search strategy is an optimal mixed strategy for the Searcher in the game. It follows that the infinum in Definition~\ref{def:rho} is in fact a minimum. A mixed strategy for the Hider is a probability distribution $h$ over the vertices $\mathcal{V}-\{O\}$, and for mixed strategies $h$ and $s$ of the Hider and Searcher respectively, we write $T(s,h)$ and $\hat{T}(s,h)$ for the corresponding expected search time and expected normalized search time.

We will obtain lower bounds for $\rho(G)$ by giving explicit Hider strategies. More precisely, if $h$ is a given mixed Hider strategy, the minimax theorem implies that $\rho(G) \ge \min_{S \in \mathcal{S}} \hat{T}(S,h)$.

\section{The search ratio of weighted graphs}
\label{sec:general.discrete}

In this section we show that the problem of computing the (deterministic) search ratio is NP-hard.
We also give a search strategy that achieves a $4\ln(4) \approx 5.55$ approximation ratio.

\begin{theorem}
    Given a graph $G$ with root $O$ and a constant $R \ge 0$, it is NP-Complete to decide
    whether $\sigma(G) \leq R$.
    \label{thm:np-c}
\end{theorem}
\begin{proof}
% \proof{Proof.}
The proof is based on a reduction from 3-SAT. Given a 3-SAT instance consisting of $n$ variables and $m$ clauses
with $m\geq n$, we construct an instance of our problem.

We construct the graph $G$ consisting of  vertices $O,P$, a vertex $C_j$ for every
clause (the \emph{clause vertices}), vertices $X_i$ (the \emph{variable vertices}) and
vertices $X_i^0,X_i^1$ (the \emph{literal vertices}) for every variable.
For every $i=1,\ldots,n$ there are unit length edges of the form
$
(X_i,X_i^0),
(X_i,X_i^1),
(P,X_i^0),
(P,X_i^1).
$
For every variable $x_i$ appearing positively in the $j$-th clause there is an edge $(C_j,X_i^1)$ of length $2$ and for every variable $x_i$ appearing negatively in the $j$-th clause there is an edge $(C_j,X_i^0)$ of length $2$. For every $j=1,\ldots,m$ there is an edge $(O,C_j)$ of length $3$ and for every $i=1,\ldots,n$ there is an edge $(O,X_i)$ of length $3$. Finally, there is an edge $(O,P)$ of length $3$. We fix $R = 1 + \frac23(n+m)$.
The construction is shown in Figure~\ref{fig:NP-hard}.
\begin{figure}[htb!]
    \centerline{\includegraphics[width=0.5\textwidth]{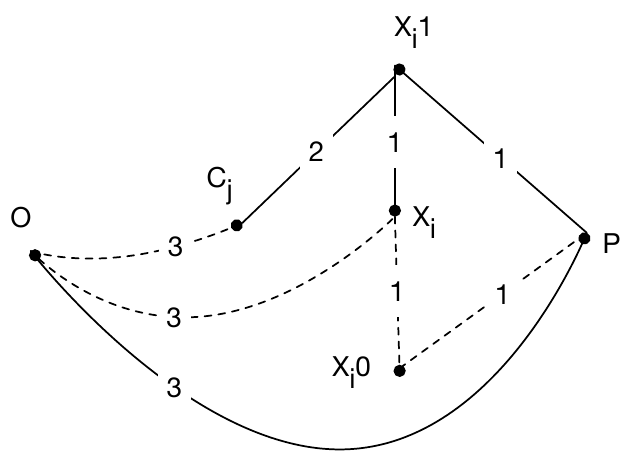}}
    \caption{A schematic view of the graph $G$ used in the reduction of Theorem~\ref{thm:np-c}.}
    \label{fig:NP-hard}
\end{figure}

Note that the vertices can be partitioned according to their distance from $O$. In particular, vertex $P$,
as well as variable and clause vertices have distances $3$, whereas literal vertices have distance $4$.

%We fix $R = 1 + \frac23(n+m)$.

% \begin{claim}
% Let $H$ be a tree rooted at $O$, spanning all distance $3$ vertices and satisfying $\lambda(H) \leq 3R$.
% We claim that $H$ can be extended to a search strategy $S$ with search ratio at most $R$.
% \end{claim}

% \begin{proof}
% For this purpose, we consider an arbitrary ordering of the edges in $H$, such that each prefix $S_v$ with $v \in H$ forms a tree containing $O$. Such an ordering could be obtained by removing repeatedly an edge to a leaf in $H$ and inversing the obtained order. Since $H$ covers all variable vertices $X_i$ it must cover in addition at least a vertex of the form $X_i^{b_i}$ for some $b_i\in\{0,1\}$.  Now complete the edge sequence by adding the remaining edges of the form $(X_i,X_i^{\overline{b_i}})$ if not already present.  The order of these final edges does not matter.  We claim that for every prefix $S_v$ of this sequence we have $\lambda(S_v)\leq R d(v)$.  This is true for $d(v)=3$, as $\lambda(S_v)\leq \lambda(H)$ by construction of the search order. And for $d(v)=4$ we observe that by $m\geq n$ we have $R \geq n$, and $S_v - H$ can contain at most $n$ edges, all of length 1.
% \end{proof}

We must show that there exists a boolean assignment to the variables satisfying all clauses if and only if the search ratio of $G$ is at most $R$.

For the easy direction of the proof, consider a boolean assignment $b\in\{0,1\}^n$ satisfying all clauses. We will show that there is a search strategy with search ratio at most $R$.  First we construct a tree covering all distance $3$ vertices with total length $3R$.  The first edge is $(O,P)$, followed by the edges $(P,X_i^{b_i})$ for every $i$. Then for every $j$, there is an edge $(X_i^{b_i},C_j)$ where $X_i^{b_i}$ corresponds to an arbitrary literal satisfying the $j$-th clause. Finally the tree also contains the edges $(X_i^{b_i},X_i)$ for all $i$.  We denote the tree constructed from $b$ by $H^b$. The total length of $H^b$ is $3+2n+2m$ which is exactly $3R$ by the choice of $R$.  To turn the tree into a search strategy $S$ we order the edges from $H^b$ as enumerated above and complete them with the edges $(P,X_i^{1-{b_i}})$ for all $i$.  We have $\sigma_S(P)=1, \sigma_S(C_j) \leq 3R/3, \sigma_S(X_i) \leq 3R/3$ and $\sigma(X_i^x)\leq (3R + n)/4\leq R$ for every $i,j,x$, which shows that the search ratio of $G$ is at most $R$.

For the hard direction, assume that there is a search strategy with search ratio at most $R$.  Let $H$ be its shortest prefix covering all distance $3$ vertices.  By the definition of the search ratio we know that $\lambda(H)\leq 3R$.  Through a sequence of transformations we turn $H$ into a tree of the form $H^b$ with
$\lambda(H^b)\leq \lambda(H)$.  This will show that $b$ is a satisfying assignment for the formula and complete the proof of the theorem.

\begin{itemize}
    \item If $(O,P)$ does not belong to $H$ we add it. This must create a cycle, containing an edge of the form $(O,v)$ with $v\neq P$. Now we remove this edge, and obtain a tree of the same length.
    \item If there is an edge $(O,C_j)$ in $H$ for some $j$, then we replace this edge by the edges $(C_j,v),(v,P)$, where $v$ is a vertex corresponding to a literal from the $j$-th clause.  Some of the added edges might already have been present.  The result is a tree of no greater length.
    \item If there is an edge of the form $(O,X_i)$ in $H$ for some $i$, then we replace this edge by the edges $(X_i,X_i^0),(X_i^0,P)$.  As a result the length of the tree decreased by 1.
    \item At this stage we know that $O$ is only connected to $P$ in the tree.
    \item If there is a vertex $C_j$ connected to several vertices $v_1,\ldots,v_k$ for $k\geq 2$, then we remove the edges $(C_j,v_1),\ldots,(C_j,v_k)$.  Hence, the tree now contains $k+1$ components, the isolated vertex $C_j$ itself and $k$ components containing each some distinct vertex $v_i$. Only one of them also contains $P$.  Without loss of generality suppose that $v_1$ and $P$ are in the same component.  Then we add $(C_j,v_1)$ back to $H$ and add edges $(v_i,P)$ for each $i=2,\ldots,k$.  This way we maintain a tree, and its length decreases by $k-1$.
    \item At this stage we know that every $C_j$ vertex is incident to exactly one length $2$ edge.  Also for every $i=1,\ldots,n$, among the vertices $\{X_i,X_i^0,X_i^1,P\}$ there are at least two edges, one adjacent to $X_i$ and one adjacent to $P$.  The last edge is necessary since otherwise there would be no connection from the vertices $\{X_i,X_i^0,X_i^1\}$ to $P$, since by the previous point we know that such a path could not go through a clause vertex.
    Let $k$ be the total number of additional edges that exist among the vertex sets $\{X_i,X_i^0,X_i^1,P\}$ over all $i=1,\ldots,n$.  Then the total length of $H$ is $3 + 2m + 2n + k$, which by assumption is at most $3R$. By the choice of $R$ we have equality and thus $k=0$. This shows that $H$ is a tree of the form $H^b$ for some $b\in\{0,1\}^n$, which is a satisfying assignment. %\Halmos
\end{itemize}
\end{proof}
\smallskip
% Alternative proof structure. We could assume that $T$ is a tree spanning all distance 3 vertices of cost at most $3R$, and minimizing the number of edges adjacent to $s$ and to $C_j$ vertices.  Then for each exchange we could show a contraction to optimality and show that $T$ is of the form $T^b$.

%We now give a search algorithm that approximates the search ratio within a constant factor.  Our algorithm is
%similar to that given in \cite{koutsoupias:fixed} in the context of pathwise search on unweighted graphs.

Using an approach similar to the doubling heuristic of \cite{koutsoupias:fixed}, we obtain a constant-approximation algorithm for computing the search ratio. It is worth pointing out that the algorithm doubles the radius, and explores the
resulting graph by computing a Steiner tree of the corresponding vertex set (in contrast to pathwise search,
in which the resulting graph is simply explored depth-first).

\begin{theorem}
    There is a polynomial-time search algorithm that approximates $\sigma(G)$ within a factor of
    $4\ln(4) +\epsilon <5.55$.
    \label{thm:deterministic.approx}
\end{theorem}
\begin{proof}
% \proof{Proof.}
% We first give lower bounds for the search ratio $\sigma$.
For any $d \ge 1 $, let $\mathcal{V}_d$ be the set of vertices of $G$ at distance no more than $d$ from the root $O$. Let $G_d$ be the subtree of minimal length in $G$ that contains all the vertices in $\mathcal{V}_d$. It is easy to see that $\sigma$ is at least $\lambda(G_d)/d$, since the time taken by any search to reach the last vertex $v$ visited in $\mathcal{V}_d$ is at least $\lambda(G_d)$, and $d(v) \leq d$.

We can view the problem of computing $G_d$ as the problem of finding a minimum-cost Steiner tree for the
set of vertices $\mathcal{V}_d$. The best known polynomial time algorithm that approximates the problem within a constant factor is that of \cite{BGRS:steiner}, which has approximation ratio
$\ln (4) +\epsilon$. Let $\hat{G}_d$ be a subtree of $G$ with total length no greater than $\ln(4) \lambda(G_d)$ that contains all the vertices in $\mathcal{V}_d$.
%The subtree $\hat{G}_d$ can be computed in polynomial time, according to the algorithm of \cite{BGRS:steiner}.

Consider the following family of search strategies $S$, for fixed $1=d_0,d_1,\ldots,d_k$, where $d_k$ is the radius of $G$. In increasing order of $j \ge 1$, search all the edges of $\hat{G}_{d_j}$ in an arbitrary order (omitting those edges that have already been searched). Suppose the Hider is at some vertex $v$ reached in the $j$th phase of the algorithm. Then $d(v)$ must be at least $d_{j-1}$, and
\[
\hat{T}(S,v) \le \frac{\sum_{i=1}^{j} \lambda(\hat{G}_{d_i})}{d_{j-1}} \le \frac{\ln(4) \sum_{i=1}^{j} \lambda({G}_{d_i})}{d_{j-1}} \le \frac{\ln(4) \sigma \sum_{i=1}^{j} d_i}{d_{j-1}}.
\]
It is optimal to choose $d_i = 2^i$ (for a proof of this, see \cite{koutsoupias:fixed}). So we obtain $\hat{T}(S,v) \le 4\ln(4) \sigma <5.55 \sigma$. %\Halmos
\end{proof}
\smallskip

\section{Trees and Unweighted Graphs}
\label{sec:trees}

In this section we present our main technical results that apply to unweighted graphs and (weighted) trees.
For both classes of graph it is easy to show that it is optimal, in the deterministic setting,
to search the vertices in non-decreasing order of their distance from the root.

If $G$ is a graph with root $O$, for any $r > 0$ let $\mathcal{V}_r$ be the set of vertices in $G$ at distance no more than $r$ from the root $O$ and let $G_r$ be the induced subgraph of $G$ with vertex set $\mathcal{V}_r$.

\begin{proposition}
    \label{prop:det.search}
    Let $G$ be a rooted graph and suppose that $G$ is a tree or an unweighted graph. Then an optimal search strategy is to search the vertices in non-decreasing order of their distance from the root. The search ratio $\sigma$ is given by
    \begin{enumerate}
        \item[(i)] $\sigma = \sup_{r>0} \frac{\lambda(G_r)}{r}$ if $G$ is a tree and
        \smallskip
        \item[(ii)] $\sigma = \sup_{r>0} \frac{\left|\mathcal{V}_r\right|-1}{r}$ if $G$ is an unweighted graph.
    \end{enumerate}
\end{proposition}
\begin{proof}
% \proof{Proof.}
Suppose a search $S$ visits the vertices in some order $v_1,\ldots,v_n$ which is not non-decreasing in order of distance from $O$ and let $v_0 = O$. Let $i \ge 1$ be minimal such that $d(v_{i-1}) < d(v_i)$ but there exists some $j > i$ such that $d(v_j) < d(v_i)$. (Such an $i$ must exist.) We may assume that $j$ is chosen so that $d(v_j)$ is minimal. Define a new search $S'$ that is the same as $S$ except the portion of the search that visits the vertices $v_i, v_{i+1},\ldots, v_j$ now visits them in the order $v_j, v_i, v_{i+1},\ldots, v_{j-1}$. By the minimality of $d(v_j)$, there must be a path from $j$ to $O$ contained in $\{v_1,\ldots, v_{i-1}\}$, so $S'$ is a feasible search.

As long as $G$ is a tree or it is unweighted, the new search time of $v_{j}$ under $S'$ is smaller, so the new normalized search time is also smaller. The new search times of the vertices $v_i,\ldots,v_{j-1}$ are larger, but no greater than the search time of $v_{j}$ under $S$; also they are further from $O$ than $v_{j}$, so their new normalized search times are no larger than that of $v_j$ under $S$. The normalized search time of every other vertex is the same under $S$ and $S'$. Thus $\sigma_{S'} \le \sigma_{S}$.

Repeating this process a finite number of times results in a search that visits the vertices in non-decreasing order of distance from $O$ and has a search ratio at most that of $S$, so is therefore optimal.

The two expressions for the search ratio of $G$ in the statement of the proposition follow immediately. %\Halmos
\end{proof}
\smallskip

Note that the above argument does not work for weighted graphs (as we should expect) since in general, swapping the order in which adjacent vertices are visited changes the search times of the other vertices.

In the randomized setting for trees and unweighted graphs we first show that the optimal deterministic search approximates the optimal randomized search by a factor of $2$. To prove this, we use the following collection of lower bounds for $\rho$. For each non-root vertex $v$ of $G$, let $\lambda_v$ denote the length of the unique edge incident to $v$ on some shortest path between $O$ and $v$ (so if $G$ is an unweighted graph then $\lambda_v = 1$) and let $\lambda_O =0$. For a set $\mathcal{A}$ of vertices, let $\lambda(\mathcal{A}) = \sum_{v \in \mathcal{A}} \lambda_v$ and let $\Delta(\mathcal{A}) = \sum_{v \in \mathcal{A}} \lambda_v d(v)$.

\begin{lemma} \label{lemma:good-lb}
    Suppose $\mathcal{A} = \{v_1,\ldots,v_m\}$ is a set of non-root vertices of $G$, and suppose the Hider chooses each $v_i \in \mathcal{A}$ with probability $p_i=\lambda_{v_i} d(v_i) / \Delta(\mathcal{A})$. Then the normalized search time $\hat{T}(S,p)$ of any search $S$ against the Hider strategy $p$ satisfies
    \[
    \hat{T}(S,p) \ge \frac{\sum_{i \le j} \lambda_{v_i} \lambda_{v_j}}{\Delta(\mathcal{A})} \ge \frac{\lambda(\mathcal{A})^2}{2 \Delta(\mathcal{A})}.
    \]
\end{lemma}
\begin{proof}
Suppose, without loss of generality, that $S$ is the search strategy that visits the vertices of $\mathcal{A}$ in the order $v_1,\ldots,v_m$. Then clearly $T(S,v_i) \ge \sum_{j \le i} \lambda_{v_j}$ for all $v_i$ so the expected value of the search ratio of $S$, under $p$ is
\[
\sum_{i=1}^m \frac{p_i T(S,v_i)}{d(v_i)} \ge \sum_{i=1}^m \frac{\lambda_{v_i} d(v_i)}{d(v_i) \Delta(\mathcal{A})} \sum_{j \le i} \lambda_{v_j} =  \frac{\sum_{i \le j} \lambda_{v_i} \lambda_{v_j}}{\Delta(\mathcal{A})}.
\]
The second inequality is trivial.
%\Halmos
\end{proof}
\smallskip

The next proposition follows directly from Lemma~\ref{lemma:good-lb}.

\begin{proposition}
    The randomized search ratio $\rho$ of a tree or unweighted graph satisfies
    $
    \sigma/2 \le \rho \le \sigma.
    $
    \label{trees:randomized}
    Hence the optimal deterministic search is a $2$-approximation of the optimal randomized search.
\end{proposition}
\begin{proof}
% \proof{Proof.}
Let $S^*$ be the optimal deterministic search strategy that searches the vertices in non-decreasing order of their distance from the root, and suppose $v$ is a vertex such that $\sigma = \sigma_{S^*} = \hat{T}(S^*,v)$. If $\mathcal{A}$ is the set of non-root vertices at distance no more than $d(v)$, then we must have $T(S^*,v) = \lambda(\mathcal{A})$, so by Lemma~\ref{lemma:good-lb},
\[
\frac{\sigma}{\rho} \le \frac{\lambda(\mathcal{A})/d(v)}{\lambda(\mathcal{A})^2/(2\Delta(\mathcal{A}))} = \frac{2 \Delta(\mathcal{A})}{d(v) \lambda(\mathcal{A})} \le 2,
\]
since $\Delta(\mathcal{A}) \le \lambda(\mathcal{A}) d(v)$.
%\Halmos
\end{proof}
\smallskip

We will show next that we can obtain improved approximations; more precisely we will present and analyze a randomized
search for trees or unweighted graphs with approximation ratio asymptotically equal to $5/4$. In the case that $G$ is an unweighted graph, we will define the search on some shortest path tree (that is, a spanning tree of $G$ comprising shortest paths from $O$ to each vertex); if $G$ is a weighted tree then we define the search on the whole of $G$. The idea of the strategy is to partition the vertices of the tree into subsets $\mathcal{V}_0,\mathcal{V}_1,\ldots$, each of which contains vertices whose distances from $O$ are within some interval $[x_i,x_{i+1}]$, where the $x_i$ are chosen randomly according to the method described later in Definition~\ref{def:deepening}. The subsets are then searched one at a time, in increasing order of distance from $O$. Note that after visiting all the vertices in $\mathcal{V}_1 \cup \ldots \cup \mathcal{V}_j$ we can contract all the edges searched so far to the root vertex $O$ and consider the problem of how to search the induced subtree $G_{i+1}$ with vertex set $\mathcal{V}_{j+1} \cup O$.

The method of searching each of these subtrees is according to a {\em random depth-first search} (or {\em RDFS}), which we define as follows. Given a set of vertices $H$ and a depth-first search $S$ of $H$ starting at the root, consider the search $S^{-1}$ which is the depth-first search of $H$ that arrives at the leaf vertices of $H$ in the reverse order to $S$. An equiprobable choice of $S$ and $S^{-1}$ is a RDFS of $H$. It is straightforward to calculate the maximum expected search time of a RDFS.

\begin{lemma} \label{lemma:RDFS}
    Under any RDFS of a rooted tree $H$, a vertex $v$ is found in expected time $(\lambda(H)+d(v))/2$.
\end{lemma}
\begin{proof}
% \proof{Proof.}
Suppose $S$ is some depth-first search of $H$. Let $A$ be the subset of edges searched by $S$ up to and including when $v$ is reached and let $B$ be the subset of edges searched by $S^{-1}$ up to and including when $v$ is reached. It is easy to see that $A \cap B$ is the unique path from $v$ to the root. If $s$ is the RDFS that chooses $S$ and $S^{-1}$ equiprobably, then the expected time $T(s,v)$ of $s$ to reach $v$ is
\begin{align*}
T(s,v) & = (\lambda(A)+\lambda(B))/2 \\
& = (\lambda(A \cup B) + \lambda(A \cap B))/2 \\
& = (\lambda(H) + d(v)/2.
\end{align*}
%\Halmos
\end{proof}
\smallskip

We can now define the randomized deepening strategy. Let $t$ be the smallest integer such that every vertex of $G$ is at distance less than $2^t$ from $O$.

\begin{definition} [\textbf{Randomized deepening strategy}]
    \label{def:deepening}
    Suppose $G$ is a tree with root $O$. For $i=1,\ldots, t$, choose some $x_i$ uniformly at random from the interval $[2^{i-1},2^i]$ and let $x_0=1$ and $x_{t+1}=2^t$. For $i=0,\ldots,t$, let $\mathcal{V}_i$ be the set of vertices of $G$ whose distance from $O$ lies in the interval $[x_i,x_{i+1})$. We call $\mathcal{V}_0,\ldots,\mathcal{V}_t$ the {\em levels} of the search, so that $\mathcal{V}_i$ is level $i$ of the search. Let $G_0$ be the induced subtree of $G$ with vertex set $\mathcal{V}_0 \cup O$ and we define $G_i, i>0$ recursively as the induced subtree with vertex set $\mathcal{V}_i \cup O$ of the graph obtained by contracting $G_0 \cup \ldots \cup G_{i-1}$ to the root $O$. The \textbf{randomized deepening strategy} performs a RDFS of each of the trees $G_i$ in the order $G_0,\ldots,G_t$.
\end{definition}

We need two straight-forward results before stating and proving the main theorem in this section, that the random deepening strategy has approximation ratio $5/4$. Let $\mathcal{A}_i$ be set of vertices of $G$ whose distance from $O$ is in the interval $[2^{i-1},2^i)$, and let $\mathcal{A}^i = \cup_{j \le i} \mathcal{A}_j$. (For the purposes of writing the proof of Theorem~\ref{thm:main-approx} we allow $i$ to take any integer value, but note that $\mathcal{A}_i$ is only non-empty for $i=1,\ldots,t$.)

\begin{lemma} \label{lemma:tech}
    The expected sum over all vertices $v \in \mathcal{A}_i \cap \mathcal{V}_{i-1}$ of edge lengths $\lambda_v$ is $2 \lambda(\mathcal{A}_i) - \Delta(\mathcal{A}_i)/2^{i-1}$.
\end{lemma}
\begin{proof}
% \proof{Proof.}
The probability that the distance $d(v)$ from $O$ of a vertex $v$ in $\mathcal{A}_i$ is less than $x_i$ is $(2^i-d(v))/2^{i-1}$. So the expected sum over all vertices $v \in \mathcal{A}_i \cap \mathcal{V}_{i-1}$ of lengths $\lambda_v$ is
\[
\sum_{v \in \mathcal{A}_i} \left( \frac{2^i - d(v)}{2^{i-1}} \right) \lambda(v) = 2 \lambda(\mathcal{A}_i) - \Delta(\mathcal{A}_i)/2^{i-1}.
\]
%\Halmos
\end{proof}
\smallskip

We also make two simple observations about the parameters $\Delta(\mathcal{A}^i)$.

\begin{lemma} \label{lemma:Delta}
    For any $i=1,\ldots,t$ we have
    \begin{enumerate}
        \item[(i)] $\Delta(\mathcal{A}^i) \le 2^i \lambda(\mathcal{A}^i)$ and
        \item[(ii)] $\frac{\lambda(\mathcal{A}^i) - \Delta(\mathcal{A}^i)/2^i}{\rho} \le 2^{i-1}$.
    \end{enumerate}
\end{lemma}
\begin{proof}
% \proof{Proof.}
Item (i) follows from the fact that every vertex in $\mathcal{A}^i$ is at distance no further than $2^i$ from $O$.

For item (ii), we use Lemma~\ref{lemma:good-lb} to bound $\rho$, giving
\[
\frac{\lambda(\mathcal{A}^i) - \Delta(\mathcal{A}^i)/2^i}{\rho} \le \frac{\lambda(\mathcal{A}^i) - \Delta(\mathcal{A}^i)/2^i}{\lambda(\mathcal{A}^i)^2/(2 \Delta(\mathcal{A}^i))} = 2 \left( \frac{\Delta(\mathcal{A}^i)}{\lambda(\mathcal{A}^i)} \right) - 2^{1-i} \left( \frac{\Delta(\mathcal{A}^i)}{\lambda(\mathcal{A}^i)} \right) ^2.
\]
The right-hand side of the expression above is a quadratic in $\Delta(\mathcal{A}^i)/\lambda(\mathcal{A}^i)$, which is easily shown to be maximized at $\Delta(\mathcal{A}^i)/\lambda(\mathcal{A}^i) = 2^{i-1}$, where it takes a value of $2^{i-1}$.
%\Halmos
\end{proof}
\smallskip

\begin{theorem} \label{thm:main-approx}
    Let $G$ be a weighted tree or an unweighted graph. Let $s$ be the randomized deepening strategy on $G$ if $G$ is a tree, or on some shortest path tree of $G$ if $G$ is an unweighted graph. Then the approximation ratio of $s$ is asymptotically $5/4$. In particular, $\rho_s \le (5/4) \rho +1/2$.
\end{theorem}
\begin{proof}
% \proof{Proof.}
First suppose $G$ is a tree. Let $v$ be a vertex that maximizes the randomized search ratio of $s$ and suppose $v$ is at distance $d$ from $O$. Let $L$ be the expected sum over all vertices $u$ in previous levels from $v$ of the lengths $\lambda_u$ plus half the expected sum over all vertices $u$ in the same level as $v$ of the lengths $\lambda_u$. By Lemma~\ref{lemma:RDFS}, the expected search time of $v$ is at most $L+d/2$. Hence
\[
\frac{\rho_s}{\rho} \le \frac{(L+d/2)/d}{\rho} \le \frac{L/d}{\rho} + 1/(2\rho).
\]
We just have to show that $L/(d \rho) \le 5/4$. Suppose $v \in \mathcal{A}_k$ for some $k$ and let $L_1,L_2,L_3$ be the contributions to $L$ from vertices in $\mathcal{A}_{k-1},\mathcal{A}_{k},\mathcal{A}_{k+1}$, respectively, so that $L = \lambda(\mathcal{A}^{k-2}) + L_1 + L_2 + L_3$. We calculate $L_1$, $L_2$ and $L_3$ separately.

For $L_1$, observe that with probability $(d-2^{k-1})/2^{k-1}$ the vertex $v$ is in level $k$, which does not intersect with any vertices in $\mathcal{A}_{k-1}$. Otherwise, $v$ is in level $k-1$ which contains some vertices of $\mathcal{A}_{k-1}$ and by Lemma~\ref{lemma:tech},
\begin{align*}
L_1 &=\left( \frac{d-2^{k-1}}{2^{k-1}} \right)\lambda(\mathcal{A}_{k-1}) + \left( \frac{2^k - d}{2^{k-1}} \right) (\lambda(\mathcal{A}_{k-1})/2 + (2\lambda(\mathcal{A}_{k-1}) - \Delta(\mathcal{A}_{k-1})/2^{k-2}))/2 \\
&=  \left(2-\frac{d}{2^k}\right)\lambda(\mathcal{A}_{k-1}) - \left(2-\frac{d}{2^{k-1}}\right)\frac{\Delta(\mathcal{A}_{k-1})}{2^{k-1}}.
\end{align*}
Similarly, for $L_3$, if $v$ is in level $k-1$ then all vertices of $\mathcal{A}_{k+1}$ are in levels after the level of $v$. Otherwise $v$ is in level $k$ which contains some vertices of $\mathcal{A}_{k+1}$. Again applying Lemma~\ref{lemma:tech},
\begin{align*}
L_3 & = \left( \frac{d}{2^{k-1}} - 1 \right) \left( \lambda(\mathcal{A}_{k+1}) - \frac{\Delta(\mathcal{A}_{k+1})}{2^{k+1}}\right).
\end{align*}
Lastly, for $L_2$, observe that a vertex $u$ in $\mathcal{A}_k$ at distance $d(u) \le d$ is in the level before $v$ if $d(u) \le x_k <d$, otherwise it is in the same level as $v$. So the contribution $u$ makes to $L_2$ is
\[
\left( \frac{d-d(u)}{2^{k-1}} \right) \cdot \lambda_u  + \left( \frac{2^{k-1} - (d-d(u))}{2^{k-1}} \right) \cdot \frac{\lambda_u}{2} = \frac{(d-d(u)+2^{k-1})\lambda_u}{2^k}.
\]
Similarly, if $d(u) > d$, then $u$ is in the level after $v$ if $d \le x_k <d(u)$, otherwise $u$ is in the same level as $v$, so the contribution $u$ makes to $L_2$ is
\[
\left( \frac{d(u)-d}{2^{k-1}} \right) \cdot 0  + \left( \frac{2^{k-1} - (d(u)-d)}{2^{k-1}} \right) \cdot \frac{\lambda}{2} = \frac{(d-d(u)+2^{k-1})\lambda_u}{2^k},
\]
which is the same.  Hence $L_2$ is given by
\[
L_2 = \sum_{u \in \mathcal{A}_{k}} \frac{(d-d(u)+2^{k-1})\lambda_u}{2^k} = \left( \frac{d}{2^k} + \frac{1}{2} \right) \lambda(\mathcal{A}_k) - \frac{\Delta(\mathcal{A}_k)}{2^k}.
\]
Using $L = \lambda(\mathcal{A}_{k-2}) + L_1 + L_2 + L_3$, combining our expressions for $L_1$, $L_2$ and $L_3$ and rearranging, we obtain
\begin{align*}
L &= \left( 1 - \frac{d}{2^k} \right) \left( \frac{\Delta(\mathcal{A}^{k-2})}{2^{k-2}} - \lambda(\mathcal{A}^{k-2}) \right)
+ \left( \frac{3}{2} - \frac{d}{2^{k-1}} \right) \left( \lambda(\mathcal{A}^{k-1}) - \frac{\Delta(\mathcal{A}^{k-1})}{2^{k-1}} \right) \\
&+ \left( \frac{3}{2} - \frac{d}{2^{k}} \right) \left( \lambda(\mathcal{A}^k) - \frac{\Delta(\mathcal{A}^k)}{2^{k}} \right)
+ \left( \frac{d}{2^{k-1}} - 1 \right) \left( \lambda(\mathcal{A}^{k+1}) - \frac{\Delta(\mathcal{A}^{k+1})}{2^{k+1}} \right).
\end{align*}
The first term in the expression on the right-hand side above is non-positive, since $d \le 2^k$ and $\Delta(\mathcal{A}^{k-2}) \le 2^{k-2} \lambda(\mathcal{A}^{k-2})$, by Lemma~\ref{lemma:Delta}(i). So, dividing by $d$, we obtain
\begin{align}
\frac{L}{d} & \le
\left( \frac{3}{2d} - \frac{1}{2^{k-1}} \right)  \left( \lambda(\mathcal{A}^{k-1}) - \frac{\Delta(\mathcal{A}^{k-1})}{2^{k-1}} \right) \nonumber
+ \left( \frac{3}{2d} - \frac{1}{2^{k}} \right) \left( \lambda(\mathcal{A}^k) - \frac{\Delta(\mathcal{A}^k)}{2^{k}} \right) \\
&   + \left( \frac{1}{2^{k-1}} - \frac{1}{d} \right) \left( \lambda(\mathcal{A}^{k+1}) - \frac{\Delta(\mathcal{A}^{k+1})}{2^{k+1}} \right). \label{eq:L/d}
\end{align}
If $2^{k-1} \le d \le 3 \cdot 2^{k-2}$ then it follows from Lemma~\ref{lemma:Delta}(i) that all three of the terms on the right-hand side of (\ref{eq:L/d}) are non-negative. Hence by Lemma~\ref{lemma:Delta}(ii),
\begin{align*}
\frac{L}{d\rho} & \le
\left( \frac{3}{2d} - \frac{1}{2^{k-1}} \right) 2^{k-2}
+ \left( \frac{3}{2d} - \frac{1}{2^{k}} \right) 2^{k-1}
+ \left( \frac{1}{2^{k-1}} - \frac{1}{d} \right) 2^k \\
& = 1+ 2^{k-3}/d \\
& \le 5/4 \mbox{ (maximized when $d = 2^{k-1}$).}
\end{align*}
If $2^k \ge d > 3 \cdot 2^{k-2}$ then the first term on the right-hand side of (\ref{eq:L/d}) is negative but the other two terms are non-negative by Lemma~\ref{lemma:Delta}(i), so by Lemma~\ref{lemma:Delta}(ii),
\begin{align*}
\frac{L}{d\rho} & \le
\left( \frac{3}{2d} - \frac{1}{2^{k}} \right) 2^{k-1}
+ \left( \frac{1}{2^{k-1}} - \frac{1}{d} \right) 2^k \\
& = 3/2- 2^{k-2}/d \\
& \le 5/4 \mbox{ (maximized when $d = 2^{k}$).}
\end{align*}
This completes the proof in the case that $G$ is a tree. If $G$ is not a tree, then we remove edges from $G$ until obtaining a shortest path tree. Note that removing these edges has no effect on $\lambda(v)$ and $d(v)$ for vertices $v$, so the lower bounds given by Lemma \ref{lemma:good-lb} remain unchanged. Hence we can apply the same argument as above, implementing the randomized deepening strategy on the shortest path tree of $G$, to obtain an approximation ratio asymptotically equal to $5/4$.
%\Halmos
\end{proof}
\smallskip

We observe that the ratio $5/4$ could be improved slightly by introducing some randomization into the definition
of $\mathcal{A}_1,\ldots,\mathcal{A}_t$: that is, we could define $\mathcal{A}_i$ as the set of all edges whose
length $d$ satisfies $2^{i-1-\theta} \le d < 2^{i-\theta}$, where $\theta$ is chosen according to some
probability distribution on $[0,1]$. This would improve the approximation ratio from $5/4$, but only marginally.

We may compare this result with analogous results from \cite{koutsoupias:fixed} in the context of pathwise search: for unweighted graphs they obtain algorithms that approximate the deterministic and randomized search ratio within a factor of 6 and 8.98 respectively. For expanding search, we easily obtain the optimal strategy in the deterministic case, and we obtain a 5/4-approximate strategy in the randomized case. Although our randomized strategy is somewhat more sophisticated, the difference must be in part due to the fact that expanding search is more straightforward to deal with than pathwise search.

% We also note than in the case that $G$ is a star, a careful analysis of the proof of Theorem~\ref{thm:main-approx} shows that the ``asymptotically'' in the statement of the theorem may be removed: in other words, the randomized deepening strategy approximates the optimal randomized search on a star by a factor of $5/4$. Furthermore, it is a simple matter to extend it to a $5/4$-approximation of the randomized search ratio for the analogous pathwise search problem (since both the lower and upper estimates are almost exactly doubled).

\section{Bounding the Randomized Search Ratio}
\label{sec:star}

In this section we ask how large the randomized search ratio can be for a graph with $n$ vertices plus the root. First note that the equivalent question for the (deterministic) search ratio is easily settled. Indeed, for an arbitrary weighted graph with minimum edge length normalized to $1$, consider the search strategy that visits the vertices in non-decreasing order of their distance from the root. The normalized search time of the $j$th vertex to be visited is at most $j \le n$, so $\sigma \le n$. This is tight for a {\em uniform star}, that is a star with edges of equal length.

In contrast, it is not as easy to bound the {\em randomized} search ratio of an arbitrary weighted graph. Note that the randomized search ratio of a uniform star is $(n+1)/2$, since an optimal Searcher strategy is to search the edges in a uniformly random order and an optimal Hider strategy is to choose one of the $n$ vertices uniformly at random. In this section we will show that the uniform star is, in fact, the weighted graph with the largest randomized search ratio. 

The remainder of this section is structured as follows. We first prove the result for unweighted graphs using a simple, intuitive argument 
(see Proposition~\ref{prop:unweighted.upper.bound}). To obtain the result for general, weighted graphs, one needs to resort to some heavier machinery. 
This can be accomplished in two different ways, and we include both for completeness. First, we argue how the main result
of~\cite{Condon:2009:ADA:1497290.1497300} on the seemingly unrelated problem of {\em pipelined filter ordering} can yield the desired upper
bound (see Proposition~\ref{prop:weightedstar}). However, the resulting optimal strategy cannot be described succinctly, 
since~\cite{Condon:2009:ADA:1497290.1497300} applies an involved, flow-based approach. We give an alternative proof of 
Proposition~\ref{prop:weightedstar} that uses a much simpler and intuitive strategy (based from intuition obtained from the unweighted case), 
and a game-theoretic approach (this proof was obtained independently in the conference version of this paper~\cite{stacs-expanding}). 
Last, based on Proposition~\ref{prop:weightedstar}, we show our main result of the section, namely that every graph has randomized search 
ratio that does not exceed that of a uniform star on the same number of vertices (see Theorem~\ref{thm:upper.bound.randomized.general}).

We begin by first proving the result for unweighted graphs, as it is relatively easy in this case. 
\begin{proposition}
	Let $G$ be an unweighted graph with $n$ vertices plus the root $O$. The randomized search ratio $\rho$ of $G$ satisfies $\rho \le (n+1)/2$.
  \label{prop:unweighted.upper.bound}
\end{proposition}

\begin{proof}
	The proof is by induction (the case $n=1$ is trivial). Assume it is true for graphs with $n$ or fewer vertices plus the root, and suppose $G$ has $n+1$ vertices. We must show that $\rho(G) \le (n+2)/2$. Let $G'$ be the subgraph of $G$ formed by removing an arbitrary vertex $v^*$ of $G$ at maximum distance from $O$ and all the edges incident to $v^*$, and let $s$ be an optimal randomized search of $G'$, so that $\hat{T}(s, v) \le (n+1)/2$ for any vertex $v$ in $G'$, by the induction hypothesis.
	
	Let $s'$ be the randomized search of $G$ that follows a shortest path $P$ to $v^*$ then follows $s$ (omitting edges contained in $P$). Let $s''$ be the search that, with probability $p$ follows $s'$ and with probability $1-p$ follows $s$ then searches $v^*$, where $p=1/(2d')$ and $d'=d(v')$. Then for any vertex $v$ of $G'$,
	\[
	\hat{T}(s'',v) \le \frac{T(s,v) + pd'}{d(v)} \le \frac{n+1}{2} + \frac{d'}{2d'} = \frac{n+2}{2}.
	\]
	Also,
	\[
	\hat{T}(s'',v^*) = \frac{(1-p)(n+1) + pd'}{d'} = \left(n+\frac{3}{2}\right)\left( \frac{1}{d'} \right) - \left(\frac{n+1}{2}\right) \left(\frac{1}{d'^2} \right).
	\]
	It is easy to verify that the expression above is decreasing in $d'$, and at $d'=1$ is equal to $(n+2)/2$. Hence $\rho \le \rho_s(G) \le (n+2)/2$ and the proposition follows by induction.
\end{proof}
\smallskip

To prove the analogous result for weighted graphs, we first need to prove it for stars. The general result then follows as a corollary. On stars, the problem of finding the randomized search ratio is related to a problem in pipelined filter ordering, which we discuss in more detail in the 
the following subsection.

\subsection{Connection with pipelined filter ordering}
\label{subsec:pipelined}
In this section we discuss the connection between this work and a problem in pipelined filter ordering. As mentioned in the introduction, star search is connected to the work of \cite{Condon:2009:ADA:1497290.1497300}, who study an adversarial pipelined filter ordering problem. The model may be described as follows. An item, or {\em tuple} must be routed in some order through a set $\{O_1,\ldots, O_n\}$ of {\em operators}, each of which tests whether the tuple satisfies some predicate (or {\em filter}) of a conjunction. There is a known cost $c_i$ of each operator $O_i$ to process a tuple, and the tuple is routed through the operators until it fails one of the tests (and is eliminated) or it passes all of them. We may alternatively think of this as a product being subjected to several quality tests before being sent to the market.

In the problem of \cite{Condon:2009:ADA:1497290.1497300}, an adversary chooses the set of filters which will eliminate the tuple. The aim is to choose a randomized routing of the tuple to minimize the {\em multiplicative regret}: that is the expected ratio of the total cost of eliminating the tuple to the cost of eliminating the tuple if the adversary's choice were known {\em a priori}. The authors argue that we may as well assume the adversary chooses only one filter to eliminate the tuple, as he is not disadvantaged by this restriction. If the ordering chosen is (without loss of generality) $O_1,\ldots, O_n$ and the adversary chooses the filter corresponding to operator $O_j$, then the multiplicative regret is $(c_1+\cdots+c_j)/c_j$. It is easy to see that the problem is exactly equivalent to our problem of calculating the optimal randomized search for a star graph with edges of lengths $c_1,\ldots,c_n$.

\cite{Condon:2009:ADA:1497290.1497300} find a polynomial time algorithm for their problem, and we summarize their result in the context of the randomized search ratio of a star graph.

\begin{theorem}[From Section 4 of \cite{Condon:2009:ADA:1497290.1497300}]
    \label{thm:stars}
    Let $G$ be a star graph with root $O$ and vertices $v_1,\ldots,v_n$ at distances $c_1,\ldots,c_n$ from $O$, where the $c_i$ are non-decreasing. There is a polynomial-time search algorithm that calculates the optimal randomized search of $G$. The randomized search ratio $\rho(G)$ is given by
    \begin{align}
    \rho(G) = \max_{1 \le k \le n} \frac{\sum_{1 \le i \le j \le k} c_i c_j}{\sum_{i=1}^k c_i^2}. \label{eq:stars}
    \end{align}
\end{theorem}

Note that the right-hand side of~(\ref{eq:stars}) is the maximum over all choices of $\mathcal{A}=\{v_1,\ldots,v_k\}$ of the value of the first lower bound in Lemma~\ref{lemma:good-lb}. It follows that the optimal strategy of the Hider is to choose such a set $\mathcal{A}$ which maximizes this bound and hide at vertex $v_i \in \mathcal{A}$ with probability proportional to $c_i^2$. The optimal strategy for the Searcher cannot be described so succinctly, and we refer to \cite{Condon:2009:ADA:1497290.1497300} for details of their algorithm.

%It is interesting to ask the question, which star with $n$ edges has the greatest search ratio and which has the greatest randomized search ratio? From Proposition \ref{prop:det.search}(i), we have
%\[
%\sigma(G) = \max_{j \le n} \frac{\sum_{i \le j} c_i}{c_j} \le \max_{j \le n} \frac{j c_j}{c_j} = n.
%\]
%This upper bound is tight if and only if all edges have the same length, so this is the star with maximal search ratio. For this star, it is easy to see that the randomized search ratio $\rho$ is $(n+1)/2$, and the optimal randomized search strategy is to search the vertices in a uniformly random order. This is in fact the largest value that the randomized search ratio can take, and it is a simple exercise to prove this, by maximizing the right hand side of~(\ref{eq:stars}).

The relation here to pipelined order filtering means that our results of Section~\ref{sec:trees} apply to the following generalization of the model in \cite{Condon:2009:ADA:1497290.1497300}. Suppose the order that the tuple may be routed through the operators is subject to some precedence constraints. This is, there is a partial order $\prec$ on the set of operators such if $O_i \prec O_j$ then the tuple must be processed by $O_i$ before it can be processed by $O_j$. The partial order defines a directed acyclic graph, and if that graph is a tree, then the problem of finding the randomized routing of the tuple to minimize the expected regret is equivalent to our problem of finding the optimal randomized search on a tree. By Theorem~\ref{thm:main-approx}, the randomized deepening strategy has approximation ratio $5/4$ for this problem.

\subsection{Bounding the randomized search ratio for weighted graphs}

We now show how to apply the results of the previous subsection so as to bound the randomized search ratio for weighted stars; we also 
show that the bound holds for arbitrary weighted graphs.

\begin{proposition} \label{prop:weightedstar}
	Let $G$ be a weighted star with $n$ vertices plus the root. The randomized search ratio $\rho$ of $G$ satisfies $\rho \le (n+1)/2$.
\end{proposition}

\begin{proof}
	By Theorem~\ref{thm:stars}, the randomized search ratio is equal to the right-hand side of~(\ref{eq:stars}). Suppose this expression is maximized by $k=k^*$ and, without loss of generality, assume that $\sum_{i=1}^{k^*} c_i = 1$. Then
	\begin{align}
	\rho &= \frac{\sum_{1 \le i \le j \le k^*} c_i c_j}{\sum_{i=1}^{k^*} c_i^2} \nonumber\\
	&= \frac{\frac{1}{2}\left( \sum_{i = 1}^{k^*} c_i \right)^2  + \frac{1}{2} \sum_{i=1}^{k^*} c_i^2}{ \sum_{i=1}^{k^*} c_i^2} \nonumber \\
	&	=\frac{1}{2 \sum_{i=1}^{k^*} c_i^2} + \frac{1}{2}. \label{eq:starbound}
	\end{align} 
	The right-hand side of~(\ref{eq:starbound}) is maximized when all the $c_i$'s are equal to $1/k^*$, so that $\rho \le (k^*+1)/2 \le (n+1)/2$.
\end{proof}
\smallskip

Using Proposition~\ref{prop:weightedstar}, our main theorem of this section follows.
\begin{theorem}
	Let $G$ be a weighted graph with $n$ vertices plus the root $O$. The randomized search ratio $\rho$ of $G$ satisfies $\rho \le (n+1)/2$.
  \label{thm:upper.bound.randomized.general}
\end{theorem}

\begin{proof}
Let $G'$ be the star with the same vertex set as $G$, such that for any non-root vertex $v$ in $G'$, the length of the edge with endpoints $O$ and $v$ in $G'$ is equal to the distance of the shortest path from $O$ to $v$ in $G$. This ensures that the distance from the root to $v$ is the same in $G$ and $G'$. By Proposition~\ref{prop:weightedstar}, there is a randomized search $s'$ of $G'$ with randomized search ratio no more than $(n+1)/2$. We can then construct a randomized search of $G'$ such that the expected search time of every vertex in $G$ under $s$ is no more than its expected search time in $G'$ under $s'$, which is sufficient to prove the theorem.

In order to construct $s$ we will replace every deterministic search $S'$ given positive weight in $s'$ with a deterministic search $S$ of $G$ such that the search time of every vertex in $G$ under $S$ is no more than its search time in $G'$ under $S'$. Indeed, suppose such a search $S'$ visits the vertices of $G'$ in some order $v_1,\ldots,v_n$, and we construct $S$ recursively as follows. Start by setting $S$ to be empty, then while $S$ has not yet visited all the nodes of $G$, we choose the minimal $i$ such that $v_i$ has not yet been visited and add to the end of $S$ a shortest path to $v_i$ from the region of $G$ that $S$ has already visited. It is easy to see that $S$ has the desired property.
\end{proof}
\smallskip

Proposition~\ref{prop:weightedstar} can be proved in another more direct way, without resorting to the result of \cite{Condon:2009:ADA:1497290.1497300}. Namely, we can explicitly construct a randomized search strategy on a star that has search ratio at most $(n+1)/2$. We include a description of this method of proving the proposition because it is based on an intuitive strategy which may be useful in future work.

For a given star graph $G$ we inductively define a randomized search strategy $s_k$ on the star graph $G_k$ consisting of only the edges $e_1,\ldots,e_k$ with total length $\mu_k$. Having defined the strategy $s_k$, we will define $s_{k+1}$ as a randomized mix of two strategies, $s_{k+1}^+$ and $s_{k+1}^-$, which we define in Definition~\ref{def:star-strat}. The former strategy, $s_{k+1}^+$ searches the new edge $e_{k+1}$ after searching the other edges, and works well in the case that the length of $e_{k+1}$ is a lot larger than the previous edges. If $e_{k+1}$ is not too large, then it is better to search it at some random point in the middle of $s_{k}$, which corresponds to the latter strategy, $s_{k+1}^-$.

\begin{definition} \label{def:star-strat} Suppose $s_k$ has been defined for some $k=1,\ldots,n-1$. Let $s_{k+1}^+$ and $s_{k+1}^-$ be randomized search strategies on $G_{k+1}$ defined by:
	\begin{enumerate}
		\item[(i)] $s_{k+1}^+$: follow the strategy $s_k$ on $G_k$ and then search edge $e_{k+1}$.
		\item[(ii)] $s_{k+1}^-$: choose a time $t$ uniformly at random in $[0,\mu_k]$ and denote the edge that is being searched at time $t$ by $e$. Follow the strategy $s_k$, but search edge $e_{k+1}$ immediately before searching $e$.
	\end{enumerate}
\end{definition}
\smallskip

Before giving the precise definition of $s_k$, we estimate the normalized expected search times $\hat{T}(s_{k+1}^+,v_i)$ and $\hat{T}(s_{k+1}^-,v_i)$ in terms of $\rho_{s_k}$ for the vertices $v_i$ with $i=1,\ldots,k+1$.

First suppose $i \le k$. Then clearly $\hat{T}(s_{k+1}^+,v_i) \le \rho_{s_k}$ (with equality for some $i \le k$) by definition of $\rho_{s_k}$. Under $s_{k+1}^-$, with probability $T(s_k,v_i)/\mu_k$ edge $e_{k+1}$ is searched before $e_i$, so the expected search time of $v_i$ is $T(s_k,v_i)+(T(s_k,v_i)/\mu_k)d_{k+1}$. Hence
\begin{align*}
\hat{T}(s_{k+1}^-,v_i) &= \frac{T(s_k,v_i)+(T(s_k,v_i)/\mu_k)d_{k+1}}{d_i} \\
&= \hat{T}(s_k,v_i)(1+d_{k+1}/\mu_k) \\
& \le \rho_{s_k}(1+d_{k+1}/\mu_k).
\end{align*}
Now suppose $i=k+1$. Under $s_{k+1}^+$, the time taken to find the Hider is $\mu_k+d_{k+1}$, so $\hat{T}(s_{k+1}^+,v_{k+1}) = \mu_k/d_{k+1}+1$. Under $s_{k+1}^-$, the expected search time is $\mu_k/2+d_{k+1}$ minus a random correction error which depends upon which edge $e$ is being searched under $s_k$ at the random time $t$ chosen uniformly in
$[0,\mu_k]$. The edge $e$ is $e_i$ with probability $d_i/\mu_k$, and in this case the expected value of the correction error is $d_i/2$. Hence the expected value of this correction error is $\sum_{i=1}^k (d_i/\mu_k)\cdot(d_i/2) = D_k/(2\mu_k)$. So we have
\begin{align*}
\hat{T}(s_{k+1}^-,v_{k+1}) &= \frac{\mu_k/2 + d_{k+1} - D_k/(2\mu_k)}{d_{k+1}} \\
&= \mu_k/(2d_{k+1}) + 1 - D_k/(2\mu_k d_{k+1}).
\end{align*}

To sum up, the expected search ratio for each combination of strategies can be bounded above by the payoffs in Table \ref{tab:game1}. We can now proceed to define $s_n$.

\begin{table}[htb!]
	\caption{Maximum value of $\hat{T}(s,v)$.}
	\label{tab:game1}
	\begin{center}
		\begin{tabular}{|l|c|c|}
			\hline
			{\bf Search} & \multicolumn{2}{c|}{{\bf Vertex, $v$}} \\
			{\bf strategy, $s$} & $v_i$ for some $i \le k$ & $v_{k+1}$\\ \hline
			$s_{k+1}^+$ & $\rho_{s_k}$ & $\mu_k/d_{k+1}+1$ \\ \hline
			$s_{k+1}^-$ & $\rho_{s_k}(1+d_{k+1}/\mu_k)$ & $\mu_k/(2d_{k+1}) + 1 - D_k/(2\mu_k d_{k+1})$ \\
			\hline
		\end{tabular}
	\end{center}
\end{table}

%\vspace{-1.1cm}
\begin{definition} Let $s_1$ be the only strategy available on $G_1$. Suppose $s_k$ has already been defined on $G_k$ for some $k=1,\ldots,n-1$. The strategy $s_{k+1}$ is an optimal mixture of $s_{k+1}^+$ and $s_{k+1}^-$ in the zero-sum game with payoff matrix given by Table \ref{tab:game1}.
\end{definition}

The search ratio of $s_n$ can be calculated recursively, since the search ratio $\rho_{s_{k+1}}$ of $s_{k+1}$ is at most the value of the game with payoff matrix given by Table \ref{tab:game1}, for each $k=1,\ldots,n-1$. We use this to show that $\rho_{s_n} \le (n+1)/2$.
\begin{theorem}
	\label{thm:star-bound}
	The randomized search ratio $\rho$ of star graph $G$ with $n$ edges is at most $(n+1)/2$, with equality if and only if all the edges have the same length.
\end{theorem}

\begin{proof}
	We have already pointed out that $\rho = (n+1)/2$ for the star whose edges all have the same length.
	To show that $\rho \le (n+1)/2$ we use induction on the number of edges to show that $\rho_{s_n} \le (n+1)/2$. It is clear that for $k=1$, we have $\rho_{s_k}=1 = (k+1)/2$, so assume that $\rho(s_k) \le (k+1)/2$ for some $k>1$ and we will show that $\rho_{k+1} \le (k+2)/2 = k/2+1$.
	
	First observe that if $d_{k+1} \ge 2\mu_k/k$ then the Searcher can ensure a payoff of no more than $k/2+1$ in the game in Table \ref{tab:game1} just by using strategy $s_{k+1}^+$. This is because the payoff $\rho_{s_k}$ against a vertex $v_i$ with $i \le k$ is no more than $(k+1)/2$ by the induction hypothesis and the payoff against $v_{k+1}$ is $\mu_k/d_{k+1}+1 \le k/2$.
	
	So assume that $d_{k+1} \le 2\mu_k/k$, and note also that $d_{k+1} \ge \mu_k/k$, since the lengths of the edges are non-decreasing and $d_{k+1}$ must be at least the average length of edges $e_1,\ldots,e_k$.
	
	By the induction hypothesis, $\rho_{s_k} \le (k+1)/2$, so the value of the game with payoff matrix given by Table \ref{tab:game1} cannot decrease if we replace $\rho_{s_k}$ with $(k+1)/2$ in the table. The value also does not decrease if we replace $-D_k$ by the maximum value it can take, which is $-\mu_k^2/k$ (that is, its value when $d_1,\ldots,d_k$ are all equal). In summary, $\rho_{s_{k+1}}$ is no more than the value of the game given in Table \ref{tab:game2}.
	
	\begin{table}[htb!]
		\begin{center}
			\caption{Upper bounds for $\hat{T}(s,v)$.}
			\begin{tabular}{|l|c|c|}
				\hline
				{\bf Search} & \multicolumn{2}{c|}{{\bf Vertex, $v$}} \\
				{\bf strategy, $s$} & $v_i$ for some $i \le k$ & $v_{k+1}$\\ \hline
				$s_{k+1}^+$ & $(k+1)/2$ & $\mu_k/d_{k+1}+1$ \\ \hline
				$s_{k+1}^-$ & $(k+1)(1+d_{k+1}/\mu_k)/2$ & $\mu_k/(2d_{k+1}) + 1 - \mu_k/(2kd_{k+1})$ \\
				\hline
			\end{tabular}
			\label{tab:game2}
		\end{center}
	\end{table}
	
	%\vspace{-0.7cm}
	By assumption, against strategy $s_{k+1}^+$, the best response of the Hider (that is, the highest payoff) is given by choosing vertex $v_{k+1}$. We show that against strategy $s_{k+1}^-$, the Hider's best response is to choose a vertex $v_i$ with $i \le k$. This follows from writing the difference, $\delta$ between the payoffs in entries $(2,1)$ and $(2,2)$ of Table \ref{tab:game2} as
	\[
	\delta=(k-1)\frac{\mu_k}{2d_{k+1}}\left(\left(\frac{k+1}{k-1}\right)\left(\frac{d_{k+1}}{\mu_k}\right)^2 + \frac{d_{k+1}}{\mu_k} - 1/k\right).
	\]
	The quadratic in $(d_{k+1}/\mu_k)$ inside the parentheses is increasing for positive values of $d_{k+1}/\mu_k$, and when $d_{k+1}/\mu_k = 1/k$ the quadratic is positive. Since $d_{k+1}/\mu_k \ge 1/k$, we must have $\delta \ge 0$.
	
	Hence the Hider does not have a dominating strategy in the game in Table \ref{tab:game2}. It is also clear that the Searcher does not have a dominating strategy, since it is better to search $e_{k+1}$ last if and only if the Hider is at some $v_{i}$ with $i \le k$. Therefore the game in Table \ref{tab:game2} has a unique equilibrium in proper mixed strategies (that is, the players both play each of their strategies with positive probability). The search ratio $\rho_{s_{k+1}}$ of $s_{k+1}$ is bounded above by the value $V$ of the game, which is easily verified to be
	\[
	V = k/2 + 1 - \frac{\frac{k}{2}(d_{k+1}/\mu_k-1/k)^2}{(d_{k+1}/\mu_k)^2+1/k}.
	\]
	This is clearly at most $k/2+1$, with equality if and only if $\mu_k/d_{k+1} = k$. The theorem follows by induction, and equality is only possible if $d_1=d_2=\ldots=d_{n}$. 
\end{proof}

\section{Conclusion}
We have undertaken an analysis of expanding search, as defined by \cite{AL:expanding}, focusing on the search ratio,
as introduced by \cite{koutsoupias:fixed} in the context of pathwise search in bounded domains.
In contrast to \cite{AL:expanding}, we have focused on computational and algorithmic issues of expanding search, an angle that is often
neglected in the analysis of search games. For general graphs, we showed that computing the search ratio is NP-hard, and we gave a $4\ln(4)$
approximation. Our main technical contribution is defining and analyzing explicit randomized search strategies that yield
significant improvements to the approximation of the randomized search ratio of trees and unweighted graphs (namely, an approximation equal to 5/4).

We believe that some of the techniques we introduced in this work can be applicable in the context of pathwise search. For instance,
we believe that a variants of the randomized strategy presented in Section~\ref{sec:trees} will result in improved randomized search
ratios for pathwise search in weighted trees.

We leave some open questions which we would like to see addressed by future work. Although we have showed that computing the search ratio of a
graph is NP-hard, we do not have an equivalent result for computing the randomized search ratio (though we suspect such a result holds).
It would be very interesting to improve upon the approximations of the search ratio and randomized search ratio for general weighted graphs;
the latter, in particular, appears to be quite a difficult problem that we believe will require the introduction of new techniques and approaches.
%A related question is whether the 5/4 approximation of the randomized search ratio for trees and unweighted graphs can be (significantly) improved.
Another direction for future work is related to the {\em continuous} model. In this model, the Hider may be located not only on a graph vertex, but
on any given point across an edge. Optimal strategies that minimize the deterministic search ratio are relatively
easy to obtain (see \cite{stacs-expanding}); in contrast, it is quite challenging to obtain strategies that improve upon the straightforward approximations of the randomized search ratio in this model.

Last, we note that the work that introduced the expanding search paradigm (\cite{AL:expanding}) raises several interesting optimization problems
concerning the average search time of vertices of a graph (assuming expanding search). In particular, one can define the {\em expanding minimum latency problem},
as the problem of minimizing the total latency of a graph, assuming an expanding search of the graph.
Is this problem NP-hard in general graphs?
If yes, can one obtain constant-factor approximations? Is the general problem in the setting in which the search time of a vertex
is weighted as hard as the unweighted variant?
Answers to the above questions will help provide an almost-complete picture of the computability and approximability of expanding search
across a variety of performance measures.

\section*{Acknowledgements}

The authors would like to thank Lisa Hellerstein for pointing out the connection between this work and the work in \cite{Condon:2009:ADA:1497290.1497300} on pipelined filter ordering.

Research supported in part by project
ANR-11-BS02-0015 ``New Techniques in Online Computation--NeTOC'', and by the FMJH program Gaspard Monge in optimization and operations research
and by EDF.

\end{document}